\documentclass[reqno]{amsart}

\usepackage[latin1]{inputenc}
\usepackage{amssymb}
\usepackage{graphicx}
\usepackage{amscd}
\usepackage[hidelinks]{hyperref}
\usepackage{color}
\usepackage{float}
\usepackage{graphics,amsmath,amssymb}
\usepackage{amsthm}
\usepackage{amsfonts}
\usepackage{latexsym}
\usepackage{epsf}
\usepackage{enumerate}
\usepackage{xifthen}
\usepackage{mathrsfs}
\usepackage{dsfont}
\usepackage{makecell}
\usepackage[FIGTOPCAP]{subfigure}
\usepackage{amsmath}
\allowdisplaybreaks[4]
\usepackage{listings}
\usepackage{etoolbox}
\usepackage{fancyhdr}

 \newtheoremstyle{mytheorem}
 {3pt}
 {3pt}
 {\slshape}
 {}
 {\bfseries}
 {.}
 { }
 {}

\numberwithin{equation}{section}

\theoremstyle{theorem}
\newtheorem{theorem}{Theorem}[section]
\newtheorem{theoremx}{Theorem}

\newtheorem{propositionx}{Proposition}

\newtheorem*{proposition*}{Proposition}

\theoremstyle{definition}

\newtheorem{remark}{Remark}[section]

\newcommand{\Keywords}[1]{\ifthenelse{\isempty{#1}}{}{\smallskip \smallskip \noindent \textbf{Keywords}. #1}}
\newcommand{\MSC}[2][2010]{\ifthenelse{\isempty{#2}}{}{\smallskip \smallskip \noindent \textbf{#1MSC}. #2}}
\newcommand{\abstractnote}[1]{\ifthenelse{\isempty{#1}}{}{\smallskip \smallskip \noindent \textsuperscript{\dag}#1}}

\makeatletter
\def\specialsection{\@startsection{section}{1}%
  \z@{\linespacing\@plus\linespacing}{.5\linespacing}%
  {\normalfont}}
\def\section{\@startsection{section}{1}%
  \z@{.7\linespacing\@plus\linespacing}{.5\linespacing}%
  {\normalfont\scshape}}
\patchcmd{\@settitle}{\uppercasenonmath\@title}{\Large\boldmath}{}{}
\patchcmd{\@settitle}{\begin{center}}{\begin{flushleft}}{}{}
\patchcmd{\@settitle}{\end{center}}{\end{flushleft}}{}{}
\patchcmd{\@setauthors}{\MakeUppercase}{\normalsize}{}{}
\patchcmd{\@setauthors}{\centering}{\raggedright}{}{}
\patchcmd{\section}{\scshape}{\large\bfseries\boldmath}{}{}
\patchcmd{\subsection}{\bfseries}{\bfseries\boldmath}{}{}
\renewcommand{\@secnumfont}{\bfseries}
\patchcmd{\@startsection}{\@afterindenttrue}{\@afterindentfalse}{}{}
\patchcmd{\abstract}{\leftmargin3pc}{\leftmargin1pc}{}{}

\def\maketitle{\par
  \@topnum\z@ 
  \@setcopyright
  \thispagestyle{empty}
  \ifx\@empty\shortauthors \let\shortauthors\shorttitle
  \else \andify\shortauthors
  \fi
  \@maketitle@hook
  \begingroup
  \@maketitle
  \toks@\@xp{\shortauthors}\@temptokena\@xp{\shorttitle}%
  \toks4{\def\\{ \ignorespaces}}
  \edef\@tempa{%
    \@nx\markboth{\the\toks4
      \@nx\MakeUppercase{\the\toks@}}{\the\@temptokena}}%
  \@tempa
  \endgroup
  \c@footnote\z@
  \@cleartopmattertags
}
\makeatother

\title[Distribution of reducible polynomials]{Distribution of reducible polynomials with a given coefficient set}

\author[S. Chern]{Shane Chern}
\address{Department of Mathematics, Pennsylvania State University, University Park, PA 16802, USA}
\email{shanechern@psu.edu}

\date{}

\begin{document}

{\footnotesize\noindent \textit{Bull. Math. Soc. Sci. Math. Roumanie (N.S.)} \textbf{60(108)} (2017), no. 2, 141--146.}

\bigskip \bigskip

\maketitle

\begin{abstract}
For a given set of integers $\mathcal{S}$, let $\mathcal{R}_n^*(\mathcal{S})$ denote the set of reducible polynomials $f(X)=a_nX^n+a_{n-1}X^{n-1}+\cdots+a_1X+a_0$ over $\mathbb{Z}[X]$ with $a_i\in\mathcal{S}$ and $a_0a_n\ne 0$. In this note, we shall give an explicit bound of $|\mathcal{R}_n^*(\mathcal{S})|$. We also present an application of this bound to reducible bivariate polynomials over $\mathbb{Z}[X,Y]$.


\MSC{Primary 11C08; Secondary 11N45.}
\end{abstract}

\section{Introduction}

Here and throughout this note, we say a polynomial is reducible if it is reducible over $\mathbb{Z}[X]$ or $\mathbb{Z}[X,Y]$. Furthermore, the notation $\mathbb{P}(F\ \mathrm{reducible})$ denotes the probability of $F$ being reducible under a given coefficient set. In a recent paper \cite{BSK2016}, L. Bary-Soroker and G. Kozma proved the following

\begin{theoremx}\label{th:BSK}
Let $F=F(X,Y)=\sum_{i,j\le n}\varepsilon_{i,j}X^iY^j$ be a bivariate polynomial of degree $n$ with random coefficients $\varepsilon_{i,j}\in\{\pm 1\}$. Then
$$\lim_{n\to\infty}\mathbb{P}(F\ \mathrm{reducible})=0.$$
\end{theoremx}

This result originates from similar distribution problems of reducible univariate polynomials, which were studied for a long period. Let the \textit{height} of a polynomial $f(X)=a_nX^n+a_{n-1}X^{n-1}+\cdots+a_1X+a_0$ with coefficients $a_i\in\mathbb{Z}$ be defined as $H(f)=\max\{|a_i|:i=0,1,\ldots,n\}$. For a fixed integer $n \ge 2$ and a real parameter $h \ge 1$, let $\mathcal{R}_n(h)$ denote the set of reducible polynomials $f(X)$ over $\mathbb{Z}$ with degree $n\ge 2$ and height $H(f)\le h$, and $\mathcal{R}_n^*(h)$ the subset of $\mathcal{R}_n(h)$ with $f(0)\ne 0$. The bound of $|\mathcal{R}_n(h)|$ given by G. Kuba \cite{Kub2009} reads
\begin{equation}\label{eq:Kuba}
h^n\le |\mathcal{R}_n(h)|\le C_nh^n\quad \text{for all $n\ge 3$ and $g\ge 1$},
\end{equation}
where $C_n > 0$ is a constant depending only on $n$. In fact, the left hand side comes directly from the reducibility of polynomials with $f(0)=0$. On the other hand, the upper bound has been studied by many authors; see, e.g., \cite{Dor1965, Gal1973, PS1998, Riv2015}. Furthermore, if we restrict that the coefficients of polynomials should be chosen from a given set $\mathcal{S}$, it is also natural to ask for the bound of number of such reducible polynomials with degree $n$, or at least the probability $p_{n,\mathcal{S}}$ of such random polynomials as $n\to\infty$; see \cite{Kon1999} for the case $\mathcal{S}=\{0,1\}$ and \cite{Som} for the case $\mathcal{S}=\{\pm 1\}$.

However, considering the notorious difficulty of proving
$$\lim_{n\to\infty}p_{n,\mathcal{S}}=0$$
for some $\mathcal{S}$, as Bary-Soroker and Kozma mentioned, they wanted to seek for a modest generalization, that is, adding one degree of freedom, or more precisely, adding one more variable --- just like that given in the above theorem.

\section{Revisit of Bary-Soroker and Kozma's proof and our main result}

Before presenting our main result, let us go back to Bary-Soroker and Kozma's proof of Theorem \ref{th:BSK}. In my personal opinion, the most crucial part of their proof is the following proposition listed as Eq. (3) of their paper.

\begin{propositionx}\label{prop:BSK}
Let
$$\Omega(n,h)=\left\{f=\sum_{i=0}^na_iX^i:a_i \text{ odd and } H(f)\le 2h-1\right\}.$$
Then there exists an absolute constant $C > 0$ such that for any $n > 1$ and $h>2$ the probability that a random uniform polynomial $f\in \Omega(n,h)$ is reducible satisfies
$$\mathbb{P}_{\Omega(n,h)}(f\ \mathrm{reducible})\le C\cdot \frac{n(\log h)^2}{h}\left(1+\frac{1}{2h}\right)^n.$$
\end{propositionx}

In view of their proof of this proposition, whose idea is due to I. Rivin \cite{Riv2015}, I note that we can even step further. Again, let $\mathcal{S}=\{s_1,s_2,\ldots,s_k\}$ be a given set of integers, and $\mathcal{S}^*=\mathcal{S}\backslash\{0\}$. We denote by $\mathcal{R}_n^*(\mathcal{S})$ the set of reducible polynomials $f(X)=a_nX^n+a_{n-1}X^{n-1}+\cdots+a_1X+a_0$ with $a_i\in\mathcal{S}$ and $a_0a_n\ne 0$. At last, let $d(n)=\sum_{d\mid n}1$ be the divisor function whose summation runs over all positive divisors of $n$. Our result is

\begin{theorem}\label{th:main}
Let $M$ be a positive integer such that
$$s_i\not\equiv s_j\bmod{M}\text{ for all }i\ne j\quad (i,j=1,2,\ldots,k).$$
Then
\begin{equation}\label{eq:main}
|\mathcal{R}_n^*(\mathcal{S})|\le 4(n-1)M^{n-2}\left(\sum_{a\in\mathcal{S}^*}d(a)\right)^2.
\end{equation}
\end{theorem}

\begin{remark}
One readily notes that a possible value of $M$ is $\max\mathcal{S}-\min\mathcal{S}+1$. However, for some $\mathcal{S}$, we could even find smaller $M$. For example, in the case of Bary-Soroker and Kozma's Proposition \ref{prop:BSK}, that is, $\mathcal{S}$ being the set of odd integers in the interval $[-2h+1,2h-1]$, they chose $M=2h+1$.
\end{remark}

\begin{proof}
We only need to slightly modify Bary-Soroker and Kozma's proof of Proposition \ref{prop:BSK}. Let $\Omega_n(\mathcal{S})$ be the set of polynomials with $a_i\in\mathcal{S}$ and $a_0a_n\ne 0$. We also fix $s,t>0$ with $s+t=n$ and $b_0,c_0,b_s,c_t\in\mathbb{Z}$ with $a_0=b_0c_0$ and $a_n=b_sc_t$ where $a_0,a_n\in\mathcal{S}^*$. Now we need to count the set $V=V(s,t,b_0,b_s,c_0,c_t)$ containing all polynomials $f\in\Omega_n(\mathcal{S})$ such that $f=pq$ with $\deg p=s$, $\deg q=s$, $p(0)=b_0$, $q(0)=c_0$, and leading coefficients of $p$ and $q$ being $b_s$ and $c_t$, respectively. This implies
$$|\mathcal{R}_n^*(\mathcal{S})|\le \sum_{a_0,a_n}\sum_{b_0\mid a_0,b_s\mid a_n}\sum_{s+t=n}|V(s,t,b_0,b_s,a_0/b_0,a_n/b_s)|.$$

Next we bound $|V(s,t,b_0,b_s,c_0,c_t)|$. The method is essentially the same as that of Bary-Soroker and Kozma. We consider the map $\phi:\Omega_n(\mathcal{S})\to \mathbb{Z}/M\mathbb{Z}[X]$ with
$$\phi(f)\equiv f\bmod{M}$$
for $f\in\Omega_n(\mathcal{S})$. Since $s_i\not\equiv s_j\bmod{M}$ for all $i\ne j$ $(i,j=1,2,\ldots,k)$, it follows that $\phi$ is injective. For any $\bar{p}$ (resp. $\bar{q}$) in $\mathbb{Z}/M\mathbb{Z}[X]$ with $\deg \bar{p}=s$ (resp. $\deg \bar{q}=t$), $\bar{p}(0)\equiv b_0\bmod{M}$ (resp. $\bar{q}(0)\equiv c_0\bmod{M}$), and leading coefficient $\bar{b}_s\equiv b_s\bmod{M}$ (resp. $\bar{c}_t\equiv c_t\bmod{M}$), we claim that the pair $(\bar{p},\bar{q})$ will identify at most one $f\in V(s,t,b_0,b_s,c_0,c_t)$ through the relation
$$\phi(\bar{p}\bar{q})=\phi(f),$$
since $\phi$ is injective. On the other hand, for any $f\in V(s,t,b_0,b_s,c_0,c_t)$ with $f=pq$, we can always find a pair $(\bar{p},\bar{q})=(\phi(p),\phi(q))$ such that
$$\phi(\bar{p}\bar{q})=\phi(f).$$
We therefore conclude that
$$|V(s,t,b_0,b_s,c_0,c_t)|\le \sum_{(\bar{p},\bar{q})}1=M^{s-1}M^{t-1}=M^{n-2}.$$

To complete our proof, we have
\begin{align*}
|\mathcal{R}_n^*(\mathcal{S})|&\le \sum_{a_0,a_n}\sum_{b_0\mid a_0,b_s\mid a_n}\sum_{s+t=n}|V(s,t,b_0,b_s,a_0/b_0,a_n/b_s)|\\
&\le (n-1)M^{n-2} \sum_{a_0,a_n}\sum_{b_0\mid a_0,b_s\mid a_n}1\\
&= (n-1)M^{n-2} \left(2\sum_{a\in\mathcal{S}^*}d(a)\right)^2.
\end{align*}
\end{proof}

It is also noteworthy to mention Kuba's bound \eqref{eq:Kuba}. In fact, he counted the set
$$\mathcal{P}_n^*(h)=\left\{(p,q)\in(\mathbb{Z}[X]\backslash\mathbb{Z})^2:\deg p+\deg q=n \text{ and }H(p)H(q)\le e^nh\right\}.$$
Comparing with our proof, in which we restrict the coefficients of $p$ and $q$ to $\mathbb{Z}/M\mathbb{Z}$, we conclude that Kuba's bound works better for $n=o(\log h)$.

\section{An application of Theorem \ref{th:main}}

We first step back to the last step of Bary-Soroker and Kozma's proof. As they showed in their Section 3, by substituting $Y = 2$ in $F(X, Y)$, they got
\begin{equation}\label{eq:3.1}
F(X,2)=\sum_{i=0}^n\left(\sum_{j=0}^n \pm 2^j\right)X^i.
\end{equation}
Now they only need to use the straightfoward argument that if $F(X,Y)$ is reducible, then either of the following holds: 1) $F(X,2)$ is reducible; 2) $F(2,Y)$ is reducible; 3) $F(X,Y)=f(X)g(Y)$ for some polynomials $f$ and $g$.

At a glimpse of the inner summation of the right hand of \eqref{eq:3.1}, the following identity of Euler may immediately come to the reader's mind:
\begin{equation}
\prod_{n=0}^\infty\left(x^{-3^n}+1+x^{3^n}\right)=\sum_{n=-\infty}^\infty x^n.
\end{equation}
This identity was given in Chapter 16 of Euler's \textit{Introductio in analysin infinitorum} which is entitled ``\textit{De Partitio Numerorum}''. The reader may refer to J. Blanton's translation \cite{Eul1988} of Euler's book. In fact, one may readily prove by induction that
\begin{equation}
\prod_{n=0}^{N-1}\left(x^{-3^n}+1+x^{3^n}\right)=\sum_{n=-(3^N-1)/2}^{(3^N-1)/2} x^n;
\end{equation}
see \cite[Eq. (5.4)]{And2007}, which is also an excellent expository article describing Euler's pioneering work.

Now this identity of Euler along with Theorem \ref{th:main} immediately give

\begin{theorem}\label{th:3.1}
Let $F=F(X,Y)=\sum_{i,j\le n}\varepsilon_{i,j}X^iY^j$ be a bivariate polynomial of degree $n$ with random coefficients $\varepsilon_{i,j}\in\{0,\pm 1\}$. Then
$$\lim_{n\to\infty}\mathbb{P}(F\ \mathrm{reducible})=0.$$
\end{theorem}

\begin{proof}
We substitute $Y = 3$ in $F(X, Y)$. Then
\begin{equation}\label{eq:3.4}
F(X,3)=\sum_{i=0}^n\left(\sum_{j=0}^n \varepsilon_{i,j} 3^j\right)X^i,
\end{equation}
where $\varepsilon_{i,j}\in\{0,\pm 1\}$. Thanks to Euler's identity, we immediately see that the right hand side of \eqref{eq:3.4} consists of all integer coefficient polynomials with degree $\le n$ and height $\le (3^{n+1}-1)/2=h^*$. Note also that the number of such polynomials with $a_0a_n=0$ is less than $2(2h^*+1)^n$. This implies that we only need to consider the probability $\mathbb{P}(f\ \mathrm{reducible})$ where $f$ is a random integer coefficient polynomial with $\deg f=n$, $H(f)\le h^*$, and $f(0)\ne 0$. Now by Theorem \ref{th:main}, we have
$$|\mathcal{R}_n^*(h^*)|\le 4(n-1)(2h^*+1)^{n-2}\left(2\sum_{n=1}^{h^*}d(n)\right)^2,$$
where we put $M=2h^*+1$. Hence
$$\mathbb{P}(F(X,3)\ \mathrm{reducible})\ll \frac{|\mathcal{R}_n^*(h^*)|}{(2h^*+1)^{n+1}}\ll \frac{n^3}{3^n}\quad (n\to\infty).$$
Here we use the approximation
$$\sum_{n\le x}d(x)\sim x\log x\quad (x\to\infty).$$

At last, similar to Bary-Soroker and Kozma's argument, we notice that if $F(X,Y)$ is reducible, then either of the following holds: 1) $F(X,3)$ is reducible; 2) $F(3,Y)$ is reducible; 3) $F(X,Y)=f(X)g(Y)$. We also have
$$\mathbb{P}(F(X,Y)=f(X)g(Y))\le \frac{3^{n+1}\cdot 3^{n+1}}{3^{(n+1)^2}}\ll 3^{-n^2} \quad (n\to\infty),$$
since both $f$ and $g$ have coefficients in $\{0,\pm 1\}$. Hence
\begin{align*}
\mathbb{P}(F(X,Y)\ \mathrm{reducible})\ll \frac{n^3}{3^n}\to 0\quad (n\to\infty).
\end{align*}
This ends our proof.
\end{proof}

\bibliographystyle{amsplain}

\begin{thebibliography}{99}

\bibitem{And2007}
G. E. Andrews, Euler's ``De Partitio numerorum'', \textit{Bull. Amer. Math. Soc. (N.S.)} \textbf{44} (2007), no. 4, 561--573.

\bibitem{BSK2016}
L. Bary-Soroker and G. Kozma, Is a bivariate polynomial with $\pm 1$ coefficients irreducible? Very likely! \textit{Int. J. Number Theory}, in press.

\bibitem{Dor1965}
K. D\"orge, Absch\"atzung der anzahl der reduziblen polynome, \textit{Math. Ann.} \textbf{160} (1965) 59--63.

\bibitem{Eul1988}
L. Euler, \textit{Introduction to analysis of the infinite. Book I.}, transl. by John D. Blanton, Springer-Verlag, New York, 1988. xvi+327 pp.

\bibitem{Gal1973}
P. X. Gallagher, The large sieve and probabilistic Galois theory, \textit{Analytic number theory (Proc. Sympos. Pure Math., Vol. XXIV, St. Louis Univ., St. Louis, Mo., 1972)}, pp. 91--101. Amer. Math. Soc., Providence, R.I., 1973.

\bibitem{Kon1999}
S. V. Konyagin, On the number of irreducible polynomials with $0$, $1$ coefficients, \textit{Acta Arith.} \textbf{88} (1999), no. 4, 333--350. 

\bibitem{Kub2009}
G. Kuba, On the distribution of reducible polynomials, \textit{Math. Slovaca} \textbf{59} (2009), no. 3, 349--356.

\bibitem{PS1998}
G. P\'olya and G. Szeg\"o, \textit{Problems and theorems in analysis. II. Theory of functions, zeros, polynomials, determinants, number theory, geometry}, transl. by C. E. Billigheimer. Reprint of the 1976 English translation. \textit{Classics in Mathematics}, Springer-Verlag, Berlin, 1998. xii+392 pp.

\bibitem{Riv2015}
I. Rivin, Galois groups of generic polynomials, \textit{Preprint} (2015), arXiv:1511.06446.

\bibitem{Som}
Some guy on the street, Irreducible polynomials with constrained coefficients, \textit{MathOverflow}. Available at: \url{http://mathoverflow.net/q/7969}.

\end{thebibliography}

\end{document}